\documentclass[10pt]{amsart}
\usepackage{amssymb, amscd, amsmath, amsthm, latexsym, hyperref}
\usepackage{pb-diagram, fancyhdr, graphicx, psfrag}
\usepackage{epstopdf}
\usepackage{multirow}

\newcommand{\compactlist}{\begin{list}{$\bullet$}{\setlength{\leftmargin}{1em}}}

\def\cals{\mathcal{S}}

\newcommand{\spinc}{\ifmmode{{\mathfrak s}}\else{${\mathfrak s}$\ }\fi}
\newcommand{\spinct}{\ifmmode{{\mathfrak t}}\else{${\mathfrak t}$\ }\fi}

\newcommand{\fig}[2] { \includegraphics[scale=#1]{#2} }

\newtheorem{theorem}{Theorem}[section]
\newtheorem{lemma}[theorem]{Lemma}
\newtheorem{corollary}[theorem]{Corollary}
\newtheorem{proposition}[theorem]{Proposition}
\theoremstyle{definition}

\theoremstyle{remark}

\numberwithin{equation}{section}

\begin{document}
\title{Concordance invariants of doubled knots and blowing up}
\author{Se-Goo Kim}
\author{Kwan Yong Lee}
\thanks{This work is supported by the Basic Science Research Program through the National Research Foundation of Korea (NRF) funded by the Ministry of Education (NRF-2015R1D1A1A01058384).}
\def\subjclassname{\textup{2010} Mathematics Subject Classification}
\expandafter\let\csname subjclassname@1991\endcsname=\subjclassname \expandafter\let\csname subjclassname@2000\endcsname=\subjclassname \expandafter\let\csname subjclassname@2010\endcsname=\subjclassname

\address{Department of Mathematics and Research Institute for Basic Sciences, Kyung Hee University, Seoul 02447, Korea}


\begin{abstract}
Let $\nu$ be either the Ozsv\'{a}th--Szab\'{o} $\tau$--invariant or the Rasmussen $s$--invariant, suitably normalized. For a knot $K$, Livingston and Naik defined the invariant $t_\nu(K)$ to be the minimum of $k$ for which $\nu$ of the $k$--twisted positive Whitehead double of $K$ vanishes. They proved that $t_\nu(K)$ is bounded above by $-TB(-K)$, where $TB$ is the maximal Thurston--Bennequin number. We use a blowing up process to find a crossing change formula and a new upper bound for $t_\nu$ in terms of the unknotting number. As an application, we present infinitely many knots $K$ such that the difference between Livingston--Naik's upper bound $-TB(-K)$ and $t_\nu(K)$ can be arbitrarily large.
\end{abstract}

\maketitle


\section{Introduction}\label{intro}
Let $\nu$ be an integer valued function on knots in the 3--sphere $S^3$ satisfying the following: For all knots $K$ and $J$ and for all torus knots $T_{p,q}$ with $p,q>0$,
\begin{enumerate}
\item $\nu(K\# J)=\nu(K)+\nu(J)$,
\item $|\nu(K)|\le g_4(K)$,
\item $\nu(T_{p,q})=(p-1)(q-1)/2$,
\end{enumerate}
where $g_4(K)$ is the smooth 4--ball genus of $K$. 
If two knots $K$ and $J$ are smoothly concordant, then $K\#-J$ is a slice knot, equivalently, $g_4(K\#-J)=0$. Property (2) implies that $|\nu(K\# -J)|\le 0$ and (1) then implies that $\nu(K)=\nu(J)$. We conclude that $\nu$ is a smooth concordance invariant. Now (1) implies that $\nu$ is a homomorphism from the smooth concordance group of knots, $\mathcal{C}$, into $\mathbb{Z}$.
This $\nu$ can be either of the $\tau$--invariant of Ozsv\'{a}th--Szab\'{o}~\cite{OS2003} and Rasmussen defined from Heegaard Floer homology and negative one half the $s$--invariant of Rasmussen~\cite{Ra2010}.

Let $D_+(K,k)$ be the $k$--twisted positive Whitehead double of a knot $K$, which is depicted in Figure~\ref{TD}. Livingston and Naik~\cite{LN2006} proved that, for every knot $K$, there is an integer $t_\nu(K)$ such that $\nu(D_+(K,k))=1$ for $k<t_\nu(K)$ and $\nu(D_+(K,k))=0$ for $k\ge t_\nu(K)$. This gives rise to a knot concordance invariant $t_\nu$ associated with each knot concordance invariant $\nu$ satisfying (1)--(3) since two concordant knots have concordant $k$--twisted positive Whitehead doubles for each~$k$. Notice that the definition for $t_K$ given by Livingston and Naik is equal to $t_\nu(K)-1$. We follow the definition of Hedden~\cite{H2007}. This $t_\nu$ invariant has been studied in \cite{H2007, LN2006, Pa2017}

\begin{figure}
 \setlength{\unitlength}{1pt}
 \begin{picture}(105,96)(4,0)
  \put(0,0){\fig{1}{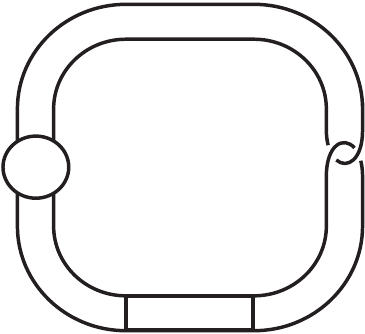}}
  \put(13.6,48){\makebox(0,0){$K$}}
  \put(57,5.5){\makebox(0,0){\small $-w+k$}}
 \end{picture}
 \caption{$D_+(K,k)$. Here $w$ is the writhe of $K$.}
 \label{TD}
\end{figure}

Livingston and Naik~\cite[Theorem 2]{LN2006} found bounds for $t_\nu$ as follows:
\begin{theorem}[Livingston--Naik]
\label{LN}
The invariant $t_\nu$ satisfies the inequality
\[TB(K) < t_\nu(K)\le -TB(-K)\] for every knot $K$, where $TB$ stands for the maximal Thurston--Bennequin number.
\end{theorem}

We say that the invariant $\nu$ is \emph{nonnegative in blowing up $+1$} or \emph{BU nonnegative} if $\nu(K)\ge 0$ for every knot $K$ that can be changed to a slice knot by a finite sequence of blowing up $+1$'s on zero linked unknots. A definition of blowing up $+1$ is given in Section~\ref{concept}.
Combining the work of Cochran--Gompf~\cite{CG1988} with the results of Ozsv\'{a}th--Szab\'{o}~\cite{OS2003} and Kronheimer--Mrowka~\cite{KM2013}, we can see that $\tau$ and $-s/2$ are BU nonnegative invariants. See Theorem~\ref{thm:nonnegative}.

We prove a crossing change formula for $t_\nu$:
\begin{theorem}
\label{mainthm}
Let $\nu$ be a BU nonnegative integer valued knot concordance invariant satisfying (1)--(3). If $K_-$ is the knot obtained from a knot $K_+$ by changing a crossing from positive to negative,  
\[
t_{\nu}(K_-) \leq t_{\nu}(K_+) \leq t_{\nu}(K_-)+4.
\]
\end{theorem}

As applications, we have
\begin{corollary}
\label{maincor}
Let $\nu$ be a BU nonnegative integer valued knot concordance invariant satisfying  (1)--(3). \\
(a) If a diagram of a knot $K$ has $p$ positive crossings and $n$ negative crossings whose crossing changes convert $K$ into a slice knot, then
\[-4n\le t_\nu(K)\le 4p.\] 
(b) There are infinitely many linearly independent concordance classes of knots $K$ such that $t_\nu(K)\le 4$ and $-TB(-K)$ can be arbitrarily large.
\end{corollary}
This shows that, if $\nu$ is either $\tau$ or $-s/2$, our upper bound for $t_\nu$ is much better than that of Livingston and Naik for infinitely many knots.

\section{Blowing up}\label{concept}

We say that a knot $J$ is constructed from a knot $K$ by a \emph{blowing up $+1$} if a diagram of $J$ is obtained from a diagram of $K$ by giving a right-handed full twist on a bunch of strings as shown in Figure~\ref{+1blowingup_4strans}(a). On the other hand, if the full twist is left-handed, this process is called a \emph{blowing up $-1$}. 

A blowing up $+1$ converting $K$ into $J$ yields a manifold pair $(W,A)$, where $W$ is a twice punctured $\mathbb{CP}^2$ and $A$ is an annulus, with boundary $\partial(W,A)=(-S^3,-J)\sqcup (S^3,K)$, as described in Figure~\ref{+1blowingup_4strans}(b): Begin with $(S^3\times I,J\times I)$, where $I=[0,1]$, choose an unknot $U$ with framing $+1$ around a bunch of strings of $J$ in $S^3\times 1$, attach a $2$--handle along~$U$, and slide the strings of $J$ over $U$ to get $K$ that is split from~$U$. Here, if the linking number of $J$ with $U$ is zero as in Figure~\ref{+1blowingup_4strans}, we call it a blowing up $+1$ \emph{on a zero linked unknot}~$U$. See~\cite{GS1994} for more details.

\begin{figure}
 $
 \begin{CD}
  \raisebox{-3ex}{\fig{1}{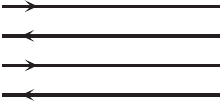}}
  @>>>
  \raisebox{-3.3ex}{
  \setlength{\unitlength}{1pt}
  \begin{picture}(64,34)(4,0)
    \put(0,0){\fig{1}{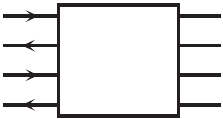}}
    \put(37,17){\makebox(0,0){$+1$}}
  \end{picture}
  }
	@=
	\raisebox{-3ex}{\fig{1}{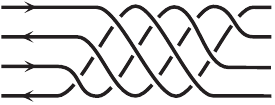}}
 \end{CD}
 $ \bigskip
 
 (a) Blowing up $+1$
 \bigskip
 
 $
 \begin{CD}
  \raisebox{-2.97ex}{
   \setlength{\unitlength}{0.9pt}
   \begin{picture}(64,34)(4,0)
    \put(0,0){\fig{0.9}{fig_4-strand_twist_box.pdf}}
    \put(37,17){\makebox(0,0){$+1$}}
   \end{picture}
  }
  @>\text{add a}>\text{$2$--handle}>
  \raisebox{-3.87ex}{
   \setlength{\unitlength}{0.9pt} 
   \begin{picture}(96,42)(4,0)
    \put(0,0){\fig{0.9}{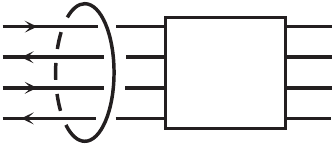}}
    \put(68,21){\makebox(0,0){$+1$}}
    \put(18,3){\makebox(0,0){\scriptsize$U$}}
    \put(39,3){\makebox(0,0){\scriptsize$+1$}}
   \end{picture}
  }
  @>\text{slide strings}>\text{over the unknot}>
  \raisebox{-2.7ex}{\fig{0.9}{fig_4-strand.pdf}}
 \end{CD}
 $ \medskip
 
 (b) Twice punctured $\mathbb{CP}^2$
 \caption{}
 \label{+1blowingup_4strans}
\end{figure}

Cochran and Gompf~\cite[Observation 2.3]{CG1988} showed that a blowing up $+1$ operation on a zero linked unknot yields an annulus $A$ smoothly properly embedded in the twice punctured $\mathbb{CP}^2$, denoted $W$, in such a way that the relative homology $[A,\partial A]$ is trivial in $H_2(W,\partial W)$.  Furthermore, if $J$ is a slice knot, then $J$ bounds a $2$--disk smoothly properly embedded in a $4$--ball, which can be glued to a boundary of $(W,A)$ to produce a once punctured $\mathbb{CP}^2$ with a smoothly properly embedded disk. We summarize:

\begin{theorem}[Cochran--Gompf] \label{CG}
If a knot $K$ can be changed to a slice knot by a finite sequence of blowing up $+1$'s on zero linked unknots, then there is a once punctured $\#_n\mathbb{CP}^2$, say $W$, and a $2$--disk $D$ smoothly properly embedded in $W$ such that $\partial D=K$ and $[D,\partial D]$ is trivial in $H_2(W,\partial W)$. Here, $\#_n\mathbb{CP}^2$ denotes the connected sum of $n$ copies of $\mathbb{CP}^2$.
\end{theorem}

The conclusion of Theorem~\ref{CG} has played an important role in studying knot concordance invariants. It together with the works of Ozsv\'{a}th--Szab\'{o}~\cite[Theorem 1.1]{OS2003} and Kronheimer--Mrowka~\cite[Corollary 1.1]{KM2013} implies the following: 

\begin{theorem}
\label{thm:nonnegative}
Both $\tau$ and $-s/2$ are BU nonnegative.
\end{theorem}
We remark that the conclusion of Theorem~\ref{CG} also implies that $K$ is $0$--positive in the sense of Cochran--Harvey--Horn~\cite{CHH2013}. They showed that more knot concordance invariants are BU nonnegative, though they do not satisfy condition (3).

Before closing this section, we prove a property of BU nonnegative knot invariants:
\begin{proposition}
\label{prop:BU nonnegative}
Let $\nu$ be a BU nonnegative integer valued knot concordance invariant satisfying (1) and (2). If a knot $K$ can be changed to a knot $J$ by a finite sequence of blowing up $+1$'s on zero linked unknots, then $\nu(K)\ge \nu(J)$.
\end{proposition}

\begin{proof}
Suppose that $K$ can be changed to a knot $J$ by a finite sequence of blowing up $+1$'s on zero linked unknots $U_1,\ldots, U_n$. We may assume that $-J$ of $K\#-J$ is contained in a small 3--ball that does not intersect any of $U_1,\ldots, U_n$. Then the same sequence of unknots $U_1,\ldots, U_n$ can be used for blowing up $+1$'s to convert $K\# -J$ into the slice knot $J\#-J$. Since $\nu$ is BU nonnegative, $\nu(K\# -J)\ge 0$ or $\nu(K)-\nu(J)\ge 0$. So, $\nu(K)\ge \nu(J)$.
\end{proof}

\section{Crossing change formula}\label{main}

We first observe a lemma:
\begin{lemma}
\label{lem:BU double}
If $K_-$ is the knot obtained from a knot $K_+$ by changing a crossing from positive to negative,
then, for each integer $k$, a blowing up $+1$ on a zero linked unknot converts $D_+(K_+,k)$ into $D_+(K_-,k)$, and another blowing up $+1$ on a zero linked unknot  converts $D_+(K_-,k-4)$ into $D_+(K_+,k)$.
\end{lemma}

\begin{proof}
If $K_+$ has writhe number $w$, then $K_-$ has writhe number $w-2$. Let $c$ be the positive crossing whose change converts $K_+$ into $K_-$. The crossing change of $c$ can be described as either blowing up $+1$ or $-1$ on an unknot as in Figure~\ref{crossing-blowing-up}. 

\begin{figure}
 \[
  \begin{CD}
   \raisebox{-2.9ex}{\fig{0.4}{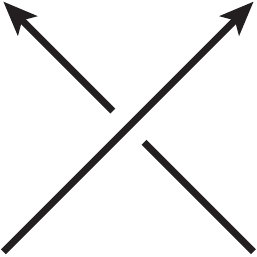}} @>>>
   \raisebox{-2.9ex}{
    \setlength{\unitlength}{0.4pt}
    \begin{picture}(102,73)(4,0)
     \put(-4,0){\fig{0.4}{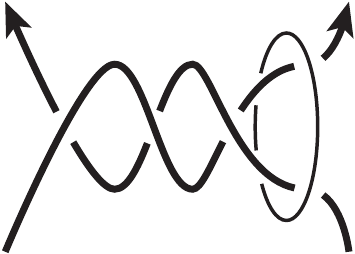}}
     \put(99,36.5){\makebox(0,0)[l]{\scriptsize $U$}}
    \end{picture}
   }  @=
   \raisebox{-2.9ex}{\fig{0.4}{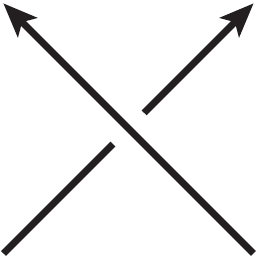}}
  \end{CD}
 \]
 (a) Blowing up $+1$ on $U$
 
 \smallskip
 \[
  \begin{CD}
   \raisebox{-2.9ex}{\fig{0.4}{fig_+crossing.pdf}} @>>>
   \raisebox{-4.05ex}{
    \setlength{\unitlength}{0.4pt}
    \begin{picture}(73,102)(4,0)
     \put(-4,0){\fig{0.4}{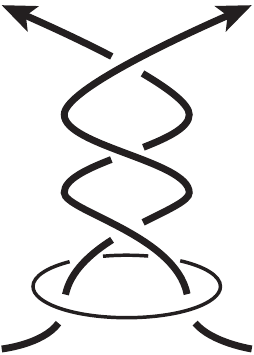}}
     \put(72,19){\makebox(0,0)[l]{\scriptsize $U$}}
    \end{picture}
   }  @=
   \raisebox{-2.9ex}{\fig{0.4}{fig_-crossing.pdf}}
  \end{CD}
 \]
 (b) Blowing up $-1$ on $U$
 \caption{Crossing change by blowing up}
 \label{crossing-blowing-up}
\end{figure}

The $k$--twisted positive Whitehead double of $K_+$, $D_+(K_+,k)$, is constructed by removing a tubular neighborhood $N$ of $K_+$ and attaching a solid torus containing the Whitehead knot back in with $k$ twists. Let $U$ be the unknot near $c$ as in Figure~\ref{crossing-blowing-up}(a), on which a blowing up process of changing $c$ takes place. If $N$ is taken sufficiently thin so that it does not meet $U$, the unknot $U$ has linking number zero with $D_+(K_+,k)$ and we can apply a blowing up $+1$ operation on $U$ of Figure~\ref{+1blowingup_4strans}(a). This gives a full twist of 4 strands, causing the effect of converting $K_+$ into $K_-$ and adding two additional positive twists on the Whitehead double, as depicted in Figure~\ref{TD_blowup}. Since $K_-$ has writhe number $w-2$, we see that the resulting knot is $D_+(K_-,k)$. This shows that $D_+(K_+,k)$ can be changed to $D_+(K_-,k)$ by a blowing up $+1$ on a zero linked unknot.

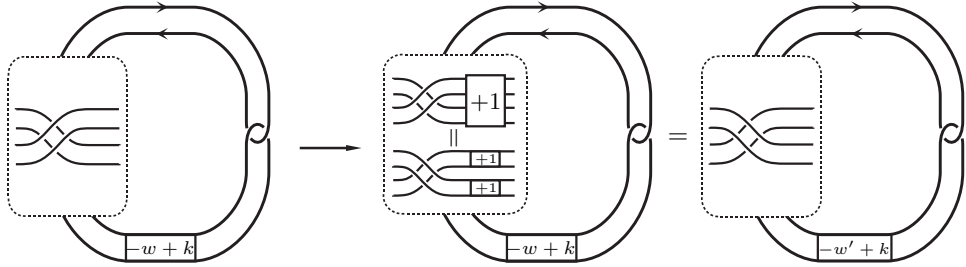
\begin{figure}
 \setlength{\unitlength}{0.85pt}
 \begin{picture}(435,122)(3,0)
  \put(0,0){\fig{0.85}{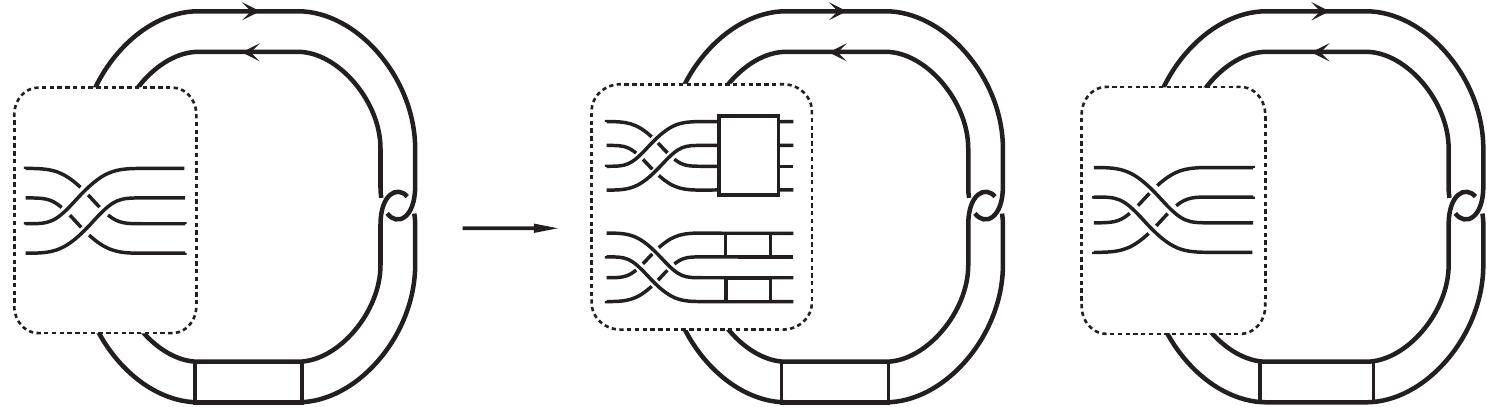}}
  \put(221,44){\makebox(0,0)[b]{\tiny $\mbox{}_{+1}$}}
  \put(221,31){\makebox(0,0)[b]{\tiny $\mbox{}_{+1}$}}
  \put(220.5,68){\makebox(0,0)[b]{\small ${+1}$}}
  \put(207,56){\makebox(0,0){\rotatebox{90}{$=$}}}
  \put(306,56){\makebox(0,0){{$=$}}}
  \put(75.3,6.7){\makebox(0,0){\scriptsize $-w+k$}}
  \put(244.8,6.7){\makebox(0,0){\scriptsize $-w+k$}}
  \put(384.3,6.7){\makebox(0,0){\tiny $-w'+k$}}
 \end{picture}
 \caption{Blowing up $+1$ of $D_+(K_+,k)$. Here $w'=w-2$.}
 \label{TD_blowup}
\end{figure}

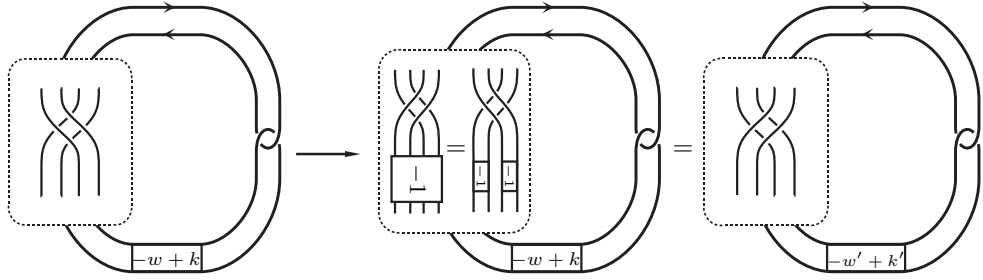
\begin{figure}
 \setlength{\unitlength}{0.85pt}
 \begin{picture}(428,117)(5,0)
  \put(0,0){\fig{0.85}{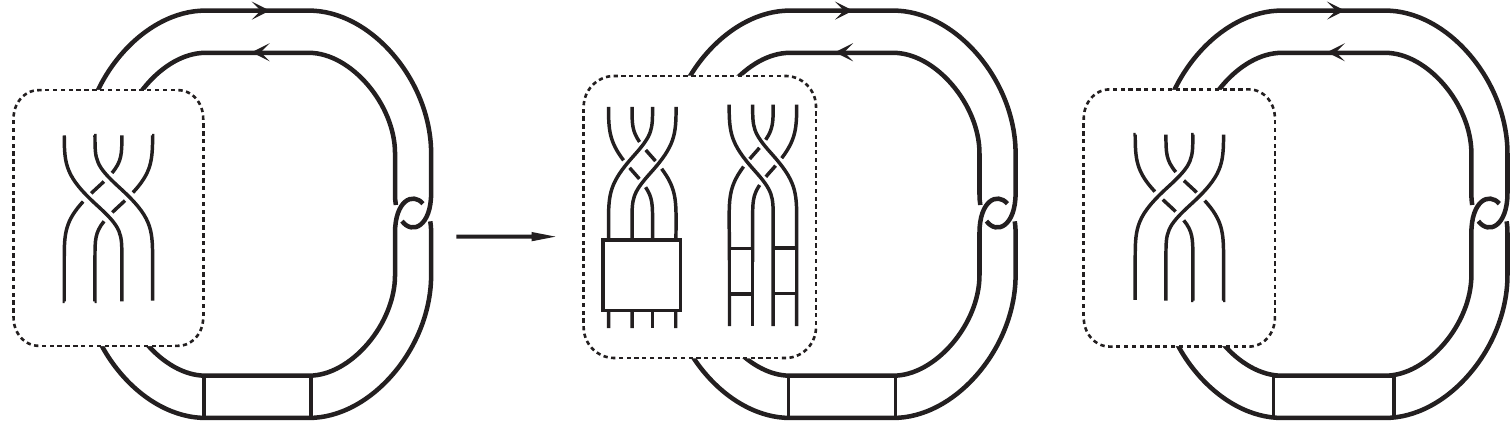}}
  \put(217.7,44){\makebox(0,0){\rotatebox{270}{\tiny $\mbox{}_{-1}$}}}
  \put(230.3,44){\makebox(0,0){\rotatebox{270}{\tiny $\mbox{}_{-1}$}}}
  \put(189.5,42.5){\makebox(0,0){\rotatebox{270}{\small ${-1}$}}}
  \put(207,56){\makebox(0,0){{$=$}}}
  \put(308,56){\makebox(0,0){{$=$}}}
  \put(77.8,7){\makebox(0,0){\scriptsize $-w+k$}}
  \put(246.8,7){\makebox(0,0){\scriptsize $-w+k$}}
  \put(389.5,7){\makebox(0,0){\tiny $-w'+k'$}}
 \end{picture}
 \caption{Blowing up $-1$ of $D_+(K_+,k)$. Here $w'=w-2$ and $k'=k-4$.}
 \label{TD_-blowup}
\end{figure}

Likewise, applying a blowing up $-1$ as in Figure~\ref{TD_-blowup}, we see $D_+(K_+,k)$ can be converted into $D_+(K_-,k-4)$ by a blowing up $-1$ on a zero linked unknot. Since a blowing up $-1$ is the reverse procedure of a blowing up $+1$, the lemma follows.
\end{proof}

\subsection*{Proof of Theorem~\ref{mainthm}}\label{MainThmProof}
Now we prove the crossing change formula for Livingston--Naik's invariant.
Since $\nu$ is BU nonnegative, by Lemma~\ref{lem:BU double} we have
\[ 
\nu(D_+(K_+,k)) \geq \nu(D_+(K_-,k)).
\]
For $k < t_{\nu}(K_-)$, $\nu(D_+(K_-,k)) = 1$ and hence $\nu(D_+(K_+,k)) = 1$ since $\nu(D_+(K_+,k))$ is $0$ or $1$ by~\cite[Theorem 2]{LN2006}. So $t_{\nu}(K_-) \leq t_{\nu}(K_+)$.

On the other hand, Lemma~\ref{lem:BU double} tells us that
\[
\nu(D_+(K_+,k)) \leq \nu(D_+(K_-,k-4)).
\]
For $k \geq t_{\nu}(K_-)+4$, $\nu(D_+(K_-,k-4)) = 0$ and hence $\nu(D_+(K_+,k))=0$, which implies $t_{\nu}(K_+) \leq t_{\nu}(K_-)+4$, as desired.

\section{Slicing numbers}\label{sec:slicing}

\subsection*{Proof of part (a) of Corollary~\ref{maincor}}\label{MainCor(a)Proof}
Suppose that a knot $K$ can be converted into a slice knot $J$ by changing $p$ positive crossings and $n$ negative crossings. Since $t_\nu(J)=0$, applying the crossing change formula repeatedly, it is easy to see that
\[
-4n \leq t_{\nu}(K) \leq 4p.
\]

\subsection*{Signed slicing numbers}
Let $K$ be a knot. Let $\cals$ be the set of ordered pairs of nonnegative integers $(p,n)$ for which a slice knot can be obtained by changing $p$ positive crossings and $n$ negative crossings of $K$. Let $\cals_+$ and $\cals_-$ be the projection of $\cals$ in the first and second coordinates, respectively. Define $u_s^\pm(K)$ to be the minimum of $\cals_\pm$, respectively. Then part (a) of Corollary~\ref{maincor} tells us that
\[
-4u_s^-(K) \leq t_{\nu}(K) \leq 4u_s^+(K).
\]
From this, we have the following observations. The \emph{slicing number} of a knot $K$, $u_s(K)$, is defined to be the minimum number of crossing changes which convert $K$ into a slice knot. Note that $u_s^\pm(K)\le u_s(K)$. Livingston~\cite{L2002} defined the invariant $U_s(K)$ by the minimum of the maximum of $(p,n)$, where $(p,n)$ runs through all elements of $\cals$. Note that $\min \{u_s^+(K), u_s^-(K)\}\le U_s(K)$. So, $t_\nu$ gives a bound for $U_s(K)$ and $u_s(K)$.

\section{Alternating knots with unknotting number one}\label{sec:alternating}
Ng~\cite{Ng2005} showed that the maximal Thurston--Bennequin number $TB(L)$ of a nonsplit alternating link $L$ is determined by its Jones Polynomial $V_L(t)$ and its classical signature $\sigma(L)$. Here, we use the convention $\sigma(T_{2,3})=-2$, which is different from that of Ng.

\begin{theorem}[Ng]
\label{Ng}
If $L$ is a nonsplit alternating link, then
\[
TB(L)=m(L) - \sigma(L)/2-1,
\]
where $m(L)$ is the minimum $t$--degree of $V_L(t)$.
\end{theorem}

The quantities $m(L)$ and $\sigma(L)$ can be computed easily from a connected, reduced, alternating diagram $D$ of $L$. Suppose that $D$ has $n$ crossings and writhe number~$w$. Let $\langle D\rangle$ be the Kauffman bracket polynomial in an indeterminate $A$. We give a checkerboard coloring to the regions of $S^2$ divided by the diagram $D$. Since $D$ is connected and alternating, it is possible to color them so that the regions incident to each crossing look like Figure~\ref{fig:checkerboard}. Let $X$ and $Y$ be the numbers of unshaded and shaded connected regions, respectively.

\begin{figure}
 \fig{0.4}{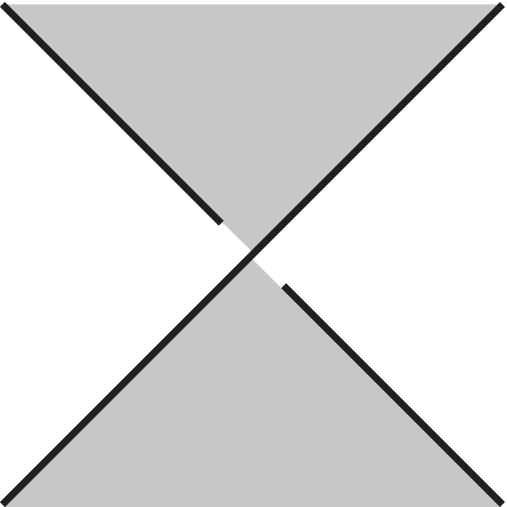}
 \caption{Checkerboard coloring}
 \label{fig:checkerboard}
\end{figure} 

Kauffman~\cite{K1987} showed that, if $D$ is connected and alternating, the bracket polynomial $\langle D \rangle$ has the maximum degree $n+2X-2$ and the minimum degree $-n-2Y+2$. Since the Jones polynomial $V_L(t)$ is equal to $(-A)^{-3w}\langle D\rangle$ when $A$ is replaced by $t^{-1/4}$, we see that the minimum $t$--degree $m(L)$ of $V_L(t)$ is equal to $m(L)=(3w-n-2X+2)/4$.

The signature $\sigma(L)$ of $L$ is $s_0-n_+ -1$, where $n_+$ is the number of positive crossings of $D$ and $s_0$ is the number of circles in the state obtained by smoothing all crossings of $D$ so that all the shaded regions become connected. This formula appears in several papers, including \cite[Proposition 3.11]{lee2005}. It is easy to see that $s_0=X$, as stated in~\cite[Proof of 2.9]{K1987}. So $\sigma(L)=X-n_+ -1$.

Combining the above two identities and noting $w=2n_+ -n$, we get the following simple formula for $TB(L)$:
\begin{corollary} 
\label{eqtnTB}
If a link $L$ has a connected, reduced, alternating diagram $D$, then
\[
TB(L)= w-X,
\]
where $w$ is the writhe of $D$ and $X$ is the number of unshaded regions in the checkerboard coloring as in Figure~\ref{fig:checkerboard}.
\end{corollary}

\subsection*{Proof of part (b) of Corollary~\ref{maincor}}\label{MainCor(b)Proof}
Let $K_{2n+1}$ be the knot which has the diagram in Figure~\ref{K(2n+1)}. Notice that $K_{2n+1}$ is isotopic to the $(-n)$--twisted double of the unknot $U$, $D_+(U,-n)$. It is alternating and has unknotting number one. One positive crossing change located at the top in the diagram is sufficient for $K_{2n+1}$ to be unknotted. By Corollary~\ref{maincor}(a), we have $t_\nu(K_{2n+1})\le 4$.
 
\begin{figure}
 \setlength{\unitlength}{1pt}
 \begin{picture}(119,115)
  \put(-2,0){\fig{1}{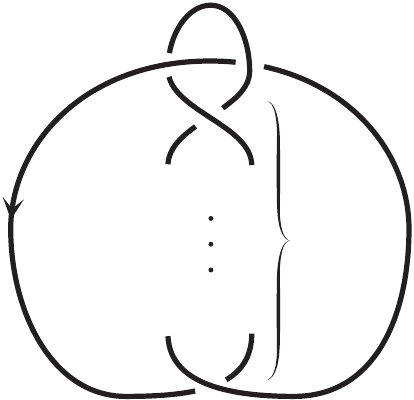}}
  \put(81,46.5){\makebox(0,0)[l]{\scriptsize $\begin{array}{c}2n-1 \\ \text{crossings}\end{array}$}}
 \end{picture}
 \caption{Knot $K_{2n+1}$ that is isotopic to $D_+(U,-n)$}
 \label{K(2n+1)}
\end{figure}

On the other hand, it is easy to see $w(-K_{2n+1})= -w(K_{2n+1})=-2n-1$ and $X(-K_{2n+1})=Y(K_{2n+1})=3$. So, by Corollary~\ref{eqtnTB}, we have $-TB(-K_{2n+1})=2n+4$.

Levine~\cite{Le1969} proved that $D_+(U,-n)$, $n>0$, represent linearly independent classes in the algebraic concordance group, and hence in the knot concordance group $\mathcal{C}$.

\subsection*{Acknowledgements} The authors would like to thank the referee for pointing out errors and for helpful suggestions regarding the exposition, which significantly improve the readability.

\bibliographystyle{plain}

\begin{thebibliography}{10}

\bibitem{CG1988}
T~D Cochran and R~E Gompf.
\newblock Applications of {Donaldson's} theorems to classical knot concordance,
  homology 3-spheres and property {P}.
\newblock {\em Topology}, 27(4):495--512, 1988.

\bibitem{CHH2013}
T~D Cochran, S~Harvey, and P~Horn.
\newblock Filtering smooth concordance classes of topologically slice knots.
\newblock {\em Geom.~ Topol.}, 17(4):2103--2162, 2013.

\bibitem{GS1994}
R~E Gompf and A~I Stipsicz.
\newblock {\em 4-manifolds and {Kirby} calculus}, volume~20.
\newblock American Mathematical Society, Providence, 1999.

\bibitem{H2007}
M~Hedden.
\newblock {Knot Floer} homology of {Whitehead} doubles.
\newblock {\em Geom.~ Topol.}, 11(4):2277--2338, 2007.

\bibitem{K1987}
L~H Kauffman.
\newblock State models and the {Jones} polynomial.
\newblock {\em Topology}, 26(3):395--407, 1987.

\bibitem{KM2013}
P~B Kronheimer and T~S Mrowka.
\newblock Gauge theory and {Rasmussen's} invariant.
\newblock {\em J.~ Topol.}, 6(3):659--674, 2013.

\bibitem{lee2005}
E~S Lee.
\newblock An endomorphism of the {Khovanov} invariant.
\newblock {\em Adv.~ Math.}, 197(2):554--586, 2005.

\bibitem{Le1969}
J~Levine.
\newblock Invariants of knot cobordism.
\newblock {\em Invent.~ Math.}, 8:98--110, 1969.

\bibitem{L2002}
C~Livingston.
\newblock The slicing number of a knot.
\newblock {\em Algebr.~Geom.~Topol.}, 2(2):1051--1060, 2002.

\bibitem{LN2006}
C~Livingston and S~Naik.
\newblock {Ozsv{\'a}th--Szab{\'o}} and {Rasmussen} invariants of doubled knots.
\newblock {\em Algebr.~ Geom.~ Topol.}, 6(2):651--657, 2006.

\bibitem{Ng2005}
L~Ng.
\newblock A {Legendrian} {Thurston}--{Bennequin} bound from {Khovanov}
  homology.
\newblock {\em Algebr.~ Geom.~ Topol.}, 5:1637--1653, 2005.

\bibitem{OS2003}
P~Ozsv\'{a}th and Z~Szab\'{o}.
\newblock Knot {Floer} homology and the four-ball genus.
\newblock {\em Geom.~ Topol.}, 7:615--639, 2003.

\bibitem{Pa2017}
J~Park.
\newblock Inequality on {$t_\nu(K)$} defined by {Livingston} and {Naik} and its
  applications.
\newblock {\em Proc.~ Amer.~ Math.~ Soc.}, 145(2):889--891, 2017.

\bibitem{Ra2010}
J~Rasmussen.
\newblock Khovanov homology and the slice genus.
\newblock {\em Invent.~ Math.}, 182(2):419--447, 2010.

\end{thebibliography}


\end{document}